\documentclass[12pt]{amsart}
\usepackage{amsfonts}
\usepackage{amsfonts,latexsym,rawfonts,amsmath,amssymb,amsthm}
\usepackage{amsmath,amssymb,amsthm,amscd,esint}
\usepackage[plainpages=false]{hyperref}
\usepackage{appendix}

\usepackage{graphicx}

\RequirePackage{color}

\numberwithin{equation}{section}

\newcommand{\beq}{\begin{equation}}
	\newcommand{\eeq}{\end{equation}}
\newcommand{\beqs}{\begin{eqnarray*}}
	\newcommand{\eeqs}{\end{eqnarray*}}
\newcommand{\beqn}{\begin{eqnarray}}
	\newcommand{\eeqn}{\end{eqnarray}}
\newcommand{\beqa}{\begin{array}}
	\newcommand{\eeqa}{\end{array}}

\newcommand{\cE}{{\mathcal E}}

\newcommand{\D}{\nabla}

\newcommand{\p}{\partial}

\newcommand{\Om}{\Omega}
\newcommand{\pom}{{\p\Om}}

\newcommand{\tr}{\triangle}

\newcommand\tbbint{{-\mkern -16mu\int}}  \newcommand\dbbint{{-\mkern -19mu\int}}   \newcommand\bbint{ {\mathchoice{\dbbint}{\tbbint}{\tbbint}{\tbbint}} }

\newtheorem{prop}{Proposition}[section]
\newtheorem{theo}[prop]{Theorem}
\newtheorem{lem}[prop]{Lemma}

\allowdisplaybreaks
\title{A priori estimate for  \\ the complex Monge-Amp\`ere equation}
\author{Jiaxiang Wang Xu-jia Wang and Bin  $\text{Zhou}^*$}

\address{Jiaxiang Wang: 
	School of Mathematical Sciences, Zhejiang University, Hangzhou 310027, China.}
\email{wangjx\underline{ }manifold@126.com}

\address{Xu-jia Wang: 
	Centre for Mathematics and Its Applications,
	The Australian National University,
	Canberra, ACT 2601.}
\email{Xu-Jia.Wang@anu.edu.au}

\address{Bin Zhou:
	School of Mathematical Sciences, Peking
	University, Beijing 100871, China.}
\email{bzhou@pku.edu.cn}

\thanks {*This research is partially supported by ARC DP 170100929 and NSFC 11571018 and 11822101.}

\subjclass[2000]{Primary: 32W20; Secondary: 35J60.}

\keywords{Complex Monge-Amp\`ere equation; Moser-Trudinger inequality; regularity.}

\begin{document}
	\maketitle
	
	\begin{abstract}
		In this paper, we use the Sobolev type inequality in \cite{WWZ} to establish  the uniform estimate and the H\"older continuity 
		for solutions to the complex Monge-Amp\`ere equation with the right-hand side in $L^p$ for any given $p>1$.
		Our proof uses various PDE techniques but not the pluri-potential theory.
	\end{abstract}
	

	\baselineskip=16.4pt
	\parskip=3pt
	\section{Introduction}
Let $\Omega$ be a bounded, smooth, strictly pseudo-convex domain in $\mathbb{C}^n$. 
Given a function $\varphi\in C^0(\partial\Omega)$ and a nonnegative function $f\in L^p(\Omega)$  for some $p>1$, 
in this paper we are concerned with the a priori estimates for solutions to the Dirichlet problem
\begin{equation}\label{cMA}
\begin{cases}
(dd^cu)^n=f\,d\mu   \ \ \text{in\ $\Omega$,}  \\
\lim_{\Omega\ni z\to z_0\in \p \Omega}u(z)=\varphi(z_0),
\end{cases}
\end{equation}  
where $\mu$ is the standard Lebesgue measure. 
For simplicity we denote the boundary condition by $u=\varphi$ on $\p \Omega$.

When $f$, $\varphi$ and $\Omega$ are smooth,
the global regularity of solutions was established in \cite{CKNS}. 
A fundamental problem to establish the a priori estimates of solutions 
when the right hand side $f\in L^p(\Omega)$ for some $p> 1$,
such as the works of De Giorgi, Nash-Moser, and Krylov-Safonov \cite{GT}.
A breakthrough was made by Ko\l{}odziej \cite{K98},  
he obtained the $L^\infty$-estimate when $f\in L^p(\Omega)$, $p> 1$. 
It was later shown that the solution is H\"older continuous on $\bar \Omega$ in \cite{GKZ}
when the domain $\Omega$ is smooth and strictly pseudo-convex, and $\varphi$ is H\"older continuous.
These results were subsequently extended to the complex Monge-Amp\`ere equation on K\"ahler manifolds \cite{DZ}.

All these results were built upon the pluri-potential theory  \cite{BT1, BT2, KL, KI, C, B98}.
In \cite{B, BGZ, L} it was asked whether there is a PDE approach to these estimates.
In this paper we prove the uniform estimate, the stability, and  the H\"older continuity of solutions 
to the complex Monge-Amp\`ere equation by PDE techniques,
and therefore give a confirmative answer to the question.

Denote by $\mathcal{PSH}(\Omega)$
the set of pluri-subharmonic functions 
and by $\mathcal{PSH}_0(\Omega)$ the set of functions in $\mathcal{PSH}(\Omega)$ 
which vanish on $\partial \Omega$. For $u\in \mathcal{PSH}_0(\Omega)\cap C^{\infty}(\bar \Omega)$, let 
be the Monge-Amp\`ere energy.
Denote
\beq
\|u\|_{\mathcal{PSH}_0(\Omega)}=[\mathcal E(u)]^{\frac{1}{n+1}},
\eeq
which is a semi-norm in the set $\mathcal{PSH}_0(\Omega)$ \cite {W1}.
In a previous paper \cite{WWZ},  the authors  proved the following 
Sobolev type inequality by a gradient flow argument.

\begin{theo}\cite{WWZ} \label{sobolev}
	Let $\Omega$ be a bounded, smooth, pseudo-convex domain.
	Then for any $p>1$,
	\beq\label{up}
	\|u\|_{L^{p}(\Omega)}\leq C \|u\|_{\mathcal{PSH}_0(\Omega)}, 
	\  \ \ \forall\ u\in \mathcal{PSH}_0(\Omega)\cap C^{\infty}(\bar \Omega) ,
	\eeq
	where $C$ depends on $n$, $p$ and $\text{diam}(\Omega)$.
\end{theo}

In \cite{WWZ}, a Moser-Trudinger type inequality was also obtained. 
Using the Sobolev type inequality  \eqref{up},
in this paper we first prove the following uniform estimate. 

\begin{theo}\label{infty}
	Assume $\varphi\in C^{\infty}(\partial\Omega)$
	and $\Om$ is a strictly pseudo-convex domain with smooth boundary.
	Let $u\in C^{\infty}(\bar\Omega)$ be a pluri-subharmonic solution 
	to \eqref{cMA}. 
	Then for any $\delta\in (0, \frac{1}{np^*})$, where $p^*=\frac{p}{p-1}$ is the conjugate of $p$ and $p>1$, 
	there is a constant $C>0$ depending on $n$, $p$, $\delta$ and $\text{diam}(\Omega)$, such that 
	\beq\label{ue}
	|\inf\limits_{\Omega}u|\leq |\inf\limits_{\p \Omega}\varphi|+C\|f\|_{L^p(\Omega)}^{\frac{1}{n}}\cdot |\Omega|^{\delta}.
	\eeq
\end{theo}

Next we prove a stability result, namely estimate \eqref{stab2} below, which was first proved in  \cite{CP, B93, K96, K02}.
Let $v$ be the solution to 
\begin{align}\label{cMA2}
\begin{cases}
(dd^cv)^n=g\,d\mu,   &\ \ \text{in\ $\Omega$},   \\
v=\psi,         &\ \ \text{on\ $\p \Omega$}, 
\end{cases}
\end{align}
where $g\in L^p(\Omega)$ with $p>1$, $\psi\in C^0(\partial\Omega)$.

\begin{theo}\label{stab}
	Let $u, v\in C^{\infty}(\bar \Omega)$ be the solutions to  \eqref{cMA} and \eqref{cMA2}, respectively. 
	Then there exists a constant $C$ depending only on $\|f\|_{L^p(\Omega)}$, 
	$\|g\|_{L^p(\Omega)}$, $n$ and $\text{diam}(\Omega)$, such that
	\beq\label{stab2}
	\|u-v\|_{L^{\infty}(\Omega)}\leq C\Big( \|f-g\|_{L^1(\Omega)}^{\frac{1}{n} 
		\frac{\delta}{1+\delta}}+\|\varphi-\psi\|^{\frac{\delta}{1+\delta}}_{L^{\infty}(\p \Omega)}\Big),
	\eeq
	where $\delta$ is the constant in Theorem \ref{infty}. 
\end{theo}

In Theorems \ref{sobolev}-\ref{stab}, we assume the solutions $u$, $v$ are sufficiently smooth. 
With the stability estimate \eqref{stab2}, 
we can also extend Theorem \ref{sobolev} to $u\in L^{\infty}_{loc}(\Omega)\cap \mathcal{PSH}_0(\Omega)$, 
and extend Theorems \ref{infty} and \ref{stab} to $u,v\in L^{\infty}_{loc}(\Omega)\cap \mathcal{PSH}(\Omega)$,
as long as  $\varphi, \psi\in C^0(\pom)$.
See Remark 3.1 for details.

The H\"older continuity of solutions  was first proved by \cite{BT1} under the assumption that 
$f^{\frac{1}{n}}\in C^{\alpha}(\Omega)$ and $\phi\in C^{2\alpha}(\partial \Omega)$. 
It was extended to the case when $f\in L^p(\Omega)$ in \cite{GKZ}. 
In this paper, we give a PDE proof for this result.

\begin{theo}\label{holder}
	Let $\Omega$ be a smooth and strictly pseudo-convex domain. 
	Assume $0\leq f\in L^p(\Omega)$ ($p>1$) and $\varphi\in C^{2\alpha}(\p\Omega)$. 
	Let $u$ be the solution to  \eqref{cMA} and
	$\hat{u}$ be the solution to $(dd^c\hat{u})^n=0$, 
	subject to the Dirichlet boundary condition $\varphi$.
	If $\triangle \hat{u}$ has finite mass in $\Omega$, then 
	$u\in C^{\alpha'}$  for any $\alpha'<\min(\alpha,\ \frac{2}{p^*n+1})$.
\end{theo}

According to \cite{GKZ}, the technical condition on $\hat u$ is satisfied when $\varphi\in C^{1,1}(\partial \Omega)$.

To obtain the a priori estimates for the complex Monge-Amp\`ere equation in Theorems \ref{infty}-\ref{holder},
we will employ various techniques developed in previous works on 
Monge-Amp\`ere type equations. 
Some refinements and improvements are needed in applying these techniques.

To prove the uniform estimate \eqref{ue} (Theorem \ref{infty}),
we use an iteration argument to establish a decay estimate \eqref{decay} 
for the Lebesgue measure of the level sets.
This iteration was used by Chou and the second author in \cite {CW} for the $k$-Hessian equation.
The third author observed that it can be improved and applied to the complex Monge-Ampere equation \cite{Z}.
Instead of the decay of the Lebesgue measure of the level sets,
Ko\l{}odziej established the decay for the capacity of level sets \cite{K98}.

The stability theorem was first proved  by directly computaion in \cite{CP} when $f, g\in L^2(\Omega)$.
For $f, g\in L^p(\Omega)$ with $p> 1$, B\l{}ocki obtained an $L^n$-$L^1$-stability theorem in \cite{B93}. Then by using capacity estimates, Ko\l{}odziej  proved the $L^\infty$-$L^1$-stability  as in Theorem \ref{stab}  above in \cite{K96, K02}.  
In our iteration proof of Theorem \ref{stab}, we replace the capacity in \cite{K96, K02} by the Lebesgue measure.  
However, since the sets $\Omega_s=\{u-v>s\}$ are not level sets anymore, 
in order to apply our Sobolev type inequality \eqref{up}, 
we will make an extension of  the domain and use an approximation argument.
The key step in the proof of  H\"older regularity (Theorem \ref{holder}) is Proposition \ref{prep-holder},
where we  also replace the capacity in  \cite{GKZ}  by the Lebesgue measure, 
and use a similar iteration argument as in Theorems \ref{infty} 
and \ref{stab}. The rest of the proof follows as in \cite{GKZ}. We will include the details of the proof for convenience of the readers.

The organization of this paper is as follows. 
In Section 2, we establish the uniform estimate. 
In Sections 3, we prove the stability of solutions.
Finally in Section 4 we prove the H\"older regularity of solutions.


\section{The uniform estimate}

In this section we consider the following Dirichlet problem, 
\begin{equation}\label{weak-cMA}
\begin{cases}
(dd^cu)^n=f\,d\mu   &\ \ \text{in\ $\Omega$,}  \\
u=\varphi              &\ \ \text{on\ $\p \Omega$,}
\end{cases}
\end{equation}
where $0\leq f\in L^p(\Omega)\bigcap C(\bar\Omega)$ and $\mu$ is the standard Lebesgue measure.  

\begin{theo}\label{weak-infty}
Assume $\varphi\in C^0(\bar \Omega)$. Let $u\in C^{\infty}(\bar\Omega)$ be a pluri-subharmonic solution to \eqref{weak-cMA}. Then for any $0<\delta<\frac{1}{np^*}$, there is a constant $C>0$ depending on $n$, $p$, $\delta$ and the upper bound of the diameter of $\Omega$, such that 
\begin{align*}
|\inf\limits_{\Omega}u|\leq |\inf\limits_{\p \Omega}\varphi|+C\|f\|_{L^p(\Omega)}^{\frac{1}{n}}\cdot |\Omega|^{\delta}.
\end{align*}
\end{theo}

\begin{proof} 
For simplicity let us assume $\|f\|_{L^p}=1$. 
Replacing the boundary function by $\inf_{\Omega}\varphi$
and using the comparison principle, it suffices to prove the estimate for the case $\varphi=0$. 
 By \eqref{weak-cMA} and \eqref{up},  we have 
\begin{eqnarray}\label{ineq-1}
\mathcal{E}(u)&= & \frac{n!}{(n+1)\pi^n}\int_{\Omega}(-u)f   \nonumber\\
&\leq & \frac{n!}{(n+1)\pi^n}\|f\|_{L^p(\Omega)}\|u\|_{L^{p^*}(\Omega)}  \\
&\leq & C |\Omega|^{\frac{1}{p^*}(1-\frac{1}{\beta})}  \|u\|_{L^{\beta p^*}(\Omega)} \nonumber \\
&\leq & C|\Omega|^{\frac{1}{p^*}(1-\frac{1}{\beta})}\|u\|_{\mathcal{PSH}_0(\Omega)}, \nonumber
\end{eqnarray}
where $p^*=\frac{p}{p-1}$ is conjugate to $p$ and $\beta>1$. It follows that 
\begin{align}\label{ineq-2}
\|u\|_{\mathcal{PSH}_0(\Omega)}\leq C|\Omega|^{\frac{1}{np^*}(1-\frac{1}{\beta})}.
\end{align}
Using \eqref{up} again,  we have 
\begin{equation}\label{integ} 
\|u\|_{L^1(\Omega)}\leq |\Omega|^{1-\frac{1}{\beta}}\|u\|_{L^{\beta}(\Omega)}\leq C|\Omega|^{1+\delta}, 
\end{equation}  
where $\delta:=\frac{1}{np^*}-\frac{1}{\beta}(1+\frac{1}{np^*})$ and $0<\delta<1$ 
when choosing $\beta>1+np^*$. This implies that for $s>0$, 
\begin{equation}\label{decay}
|\{x\in \Omega\ |\ u< -s\} | \leq C\frac{1}{s}|\Omega|^{1+\delta}.
\end{equation}


Now we proceed to the iteration argument. 
Since each connected component of $\Omega_s:=\{x\in \Omega\ |\ u< -s\}$ 
is hyperconvex and has only almost everywhere smooth boundary for almost every $s\in (0, |\inf_{\Omega}u|)$, 
the Sobolev inequality cannot apply directly. This problem can be avoided by approximation, as follows.

Choose 
$s=s_0=2^{1+\frac{1}{\delta}}C^{1+\frac{1}{\delta}}|\Omega|^{\delta}$ in \eqref{decay}. 
Then we have $|\Omega_{s_0}|\leq \frac{|\Omega|}{2^{1+\frac{1}{\delta}}C^{\frac{1}{\delta}}}\leq \frac{1}{2}|\Omega|$ due to $C>1$. For any $k\in \mathbb{Z}_+$, define 
\begin{align}\label{iteration}
s_k=s_0+\sum_{j=1}^k2^{-\delta j}|\Omega|^{\delta}, \ \ u^k:=u^{s_k},\ \ \Omega_k=\Omega_{s_k}.
\end{align}
For each $\Omega_k$, we define 
$$
f^k:=\Big\{
{\begin{split}
 &f   \ \ \ \text{in}\ \bar{\Omega}_k,\\
  &0   \ \ \ \text{on}\ \Omega\setminus \bar{\Omega}_k.
\end{split}}
$$
Let $f_j^k$ be a sequence of smooth, monotone decreasing approximation of $f^k$ such that 
$\sup_{\Omega_k}|f^k-f_j^k|\to 0$ as $j\to \infty$. 
This implies $\|f_j^k-f^k\|_{L^p(\Omega)}\to 0$ for any $p>1$, 
but we only need $\|f_j^k-f^k\|_{L^2(\Omega)}\to 0$ in order to apply \cite{CP}. 
Consider the Dirichlet problem
\begin{align}\label{app-Diri}
\begin{cases}
(dd^cv)^n=f_j^k,\ \ \ &\text{in}\ \Omega;\\
v=0,           \ \ \ &\text{on}\ \p \Omega.
\end{cases}
\end{align}
Since $u$ is a subsolution to \eqref{app-Diri} when the right-hand side is $f^k$, there exists a solution $v_j^k$ to \eqref{app-Diri}  and $\|v_j^k\|_{L^{\infty}(\Omega)}\leq C$ for some $C>0$ independent of $j$ and $k$ but depends on $\|u\|_{L^{\infty}}$.  
Moreover, $v_j^k$ is monotone increasing. Denote $v^k=\displaystyle\lim _{j\to\infty} v_j^k$. 
Then from the first inequality of \eqref{ineq-1}, 
$$
\mathcal{E}(v_j^k) 
  =\int_{\Omega}(-v_j^k)f_j^k
   \leq  \frac{n!}{(n+1)\pi^n}\|f_j^k\|_{L^p(\Omega)}\|v_j^k\|_{L^{p^*}(\Omega)}.
$$
On the other hand, by the Sobolev ineqaulity,
$$\|v_j^k\|_{L^{\beta p^*}(\Omega)}\leq C\left(\int_{\Omega}(-v_j^k)f_j^k\right)^{\frac{1}{n+1}}.$$
Letting $j\to\infty$, we obtain
$$
\int_{\Omega}(-v^k)f^k
 =\int_{\Omega_k}(-v^k)f \leq  \frac{n!}{(n+1)\pi^n}\|f\|_{L^p(\Omega_k)}\|v^k\|_{L^{p^*}(\Omega_k)}
$$
and
$$\|v^k\|_{L^{\beta p^*}(\Omega)}\leq C\left(\int_{\Omega_k}(-v^k)f\right)^{\frac{1}{n+1}}.$$
As in \eqref{ineq-1} we then obtain
\begin{align}\label{k-integ}
\|v^k\|_{L^1(\Omega_k)}\leq |\Omega_k|^{1-\frac{1}{\beta}}\|v^k\|_{L^{\beta}(\Omega)}\leq C|\Omega_k|^{1+\delta}.
\end{align}
In view of $v^k\leq u^k=u+s_k$ in $\Omega_k$, we obtain
\begin{align}\label{k-integ}
\|u^k\|_{L^1(\Omega_k)}\leq  C|\Omega_k|^{1+\delta}.
\end{align}
Note that the constants in the Sobolev inequalities depend on the upper bound of diameters of the domains. Hence the constants here are uniform for $k$. 

We claim that $|\Omega_{k+1}|\leq \frac{1}{2}|\Omega_k|$ for any $k$. 
By induction, we assume the inequality holds for $k\leq l$. By \eqref{decay} and \eqref{k-integ}, 
\begin{align*}
\frac{1}{2^{\delta (l+1)}}|\Omega|^{\delta}\cdot |\Omega_{l+1}|\leq \|u^l\|_{L^1(\Omega_l)}\leq C|\Omega_l|^{1+\delta}. 
\end{align*}
Hence
\begin{align*}
|\Omega_{l+1}|\leq &  C|\Omega_l|^{1+\delta}\frac{2^{\delta (l+1)}}{|\Omega|^{\delta}}    \\
\leq &  C\left[\left(\frac{|\Omega_0|}{2^l}\right)^{\delta}\frac{2^{\delta (l+1)}}{|\Omega|^{\delta}}\right]\cdot |\Omega_l|  \\
\leq & \left[C^{1+\delta}\frac{2^{\delta}}{s_0^{\delta}}|\Omega|^{\delta^2}\right]\cdot |\Omega_l|   
\leq \frac{1}{2}|\Omega_l |
\end{align*}
by our choice of $s_0$.

By the above claim,  the set 
$$\Big\{x\in \Omega\ \big|\ u<-s_0-\sum_{j=1}^{\infty}\left(\frac{1}{2^{\delta}}\right)^j|\Omega|^{\delta}\Big\}$$
has measure zero. Hence, 
$${\begin{split}
\|u\|_{L^{\infty}(\Omega)}
  & \leq s_0+\sum_{j=1}^{\infty}\left(\frac{1}{2^{\delta}}\right)^j|\Omega|^{\delta}\\
  & =2^{1+\frac{1}{\delta}}C^{1+\frac{1}{\delta}}|\Omega|^{\delta}+\frac{1}{2^{\delta}-1}|\Omega|^{\delta}\\
  & \leq C|\Omega|^{\delta}.
  \end{split}} $$
\vskip-20pt
\end{proof}

\noindent{\bf Remark 2.1.}
The proof of Theorem \ref{weak-infty} was first given in \cite{Z},
which is a research report in the School of Mathematical Sciences, Peking University. 
This report series has only two issues and then stopped. 
It is unavailable in other universities  either in China or overseas.
Therefore we include the details of the proof in this paper.
Here we also refine the argument to obtain  the constant $C\|f\|_{L^p(\Omega)}^{\frac{1}{n}}\cdot |\Omega|^{\delta}$ 
for later use in the H\"older regularity.



\section{Stability estimate}

In this section, we prove the stability theorem without using the pluripotential theory. 

Let $u$, $v\in C^{\infty}(\bar\Omega)\bigcap \mathcal{PSH}(\Omega)$ be the solutions to \eqref{cMA} and \eqref{cMA2}, respectively.
 Let $w$, $w_0$ be the solutions to the Dirichlet problems
\begin{align*}
\begin{cases}
(dd^cw)^n=|f-g|\,d\mu\ \ &\text{in\ $\Omega$,}  \\
w=-|\varphi-\psi|       \ \ &\text{on\ $\p \Omega$,}
\end{cases}'
\end{align*}
and
\begin{align*}
\begin{cases}
(dd^cw_0)^n=|f-g|\,d\mu\ \ &\text{in\ $\Omega$,}  \\
w_0=0       \ \ &\text{on\ $\p \Omega$,}
\end{cases}
\end{align*}
 respectively.
By the pluri-subharmonicity, 
$${\begin{split}
(dd^c(v+w))^n
 & \geq (dd^cv)^n+(dd^cw)^n \\
 & =g\,d\mu+|f-g|\,d\mu\\
 & \geq f\,d\mu =(dd^cu)^n.
 \end{split}}
$$
Then by the comparison principle, we have $u-v\geq w$ and $w\geq w_0-\sup\limits_{\p \Omega}|\varphi-\psi|$ in $\Omega$. 
Hence,
$$\|u-v\|_{L^{1}(\Omega)}\leq \|w\|_{L^{1}(\Omega)}\leq \|w_0\|_{L^1(\Omega)}+|\Omega|\cdot\|\varphi-\psi\|_{L^{\infty}(\p \Omega)}.$$
 Let $\Omega\subset B_R(0)$ for some $R>0$. 
Next we apply Theorem 2.1 in \cite{B93} 
It holds that
\beq\label{integ-stab}
{\begin{split}
\|w_0\|_{L^n(\Omega)}
  & \leq n!R^{2n}\left[\int_{\Omega}(dd^cw_0)^n\right]^{\frac{1}{n}}\\
  & =n!R^{2n}\|f-g\|_{L^1(\Omega)}^{\frac{1}{n}}.
  \end{split}}
\eeq
Therefore, we have
$$\|u-v\|_{L^1(\Omega)}\leq n!R^{2n}|\Omega|^{1-\frac{1}{n}}\cdot\|f-g\|_{L^1(\Omega)}^{\frac{1}{n}}+|\Omega|\cdot\|\varphi-\psi\|_{L^{\infty}(\p \Omega)}.$$

\vskip 10pt

Now we use an iteration argument, similarly to that in Theorem \ref{weak-infty}, to obtain the stability.
Denote 
$$t:=\big(n!R^{2n}|\Omega|^{1-\frac{1}{n}}\cdot\|f-g\|_{L^1(\Omega)}^{\frac{1}{n}}+|\Omega|\cdot\|\varphi-\psi\|_{L^{\infty}(\p \Omega)}\big)^{\frac{\delta}{1+\delta}},$$
where $\delta$ will be determined later.
For any $s>0$, denote $\Omega_s:=\{u-v>st\}$. Then it is clear that
$$st\cdot |\Omega_s|\leq \|u-v\|_{L^1(\Omega)}\leq t^{1+\frac{1}{\delta}}.$$
This implies 
\begin{align}\label{decay-stab}
|\Omega_s|\leq t^{\frac{1}{\delta}}s^{-1}. 
\end{align}
Note that $v^s:=v+st<u$ solves 
\begin{align*}
\begin{cases}
(dd^cv^s)^n=g\,d\mu\ \ & \text{in $\Omega_s$,}  \\
v^s=u                 \ \ & \text{on $\p \Omega_s$.}
\end{cases}
\end{align*}

Now we consider an upper-continuous function
\begin{align*}
g^s(x)=\Big\{
\begin{split}
 &g(x) ,    \ \ \text{\ in\ $\bar\Omega_s$,}  \\
&\ 0 ,       \ \ \ \text{\ on\ $\Omega\setminus \bar\Omega_s$.}
\end{split}
\end{align*}
Let $\{g^s_{j}\}$ be a decreasing smooth approximation of $g^s$ such that  
$\sup_{\Omega_s}|g^s_{j}-g^s|\to 0$ as $j\to \infty$. Let $\tilde v^s_j$ be the solution to 
\begin{align*}
\begin{cases}
(dd^c\tilde v^s_{j})^n=g^s_{j} \, d\mu\ \ & \text{in $\Omega$,}  \\
\tilde v^s_{j}=0                 \ \ & \text{on $\p \Omega$,}
\end{cases}
\end{align*}
where $\mu$ is the standard Lebesgue measure. 
Then we have 
$$(dd^c(\tilde v^s_{j}+u))^n\geq (dd^cv^s)^n=g\,d\mu$$ in $\Omega_s$ 
and $\tilde v^s_{j}+u\leq v^s$ on $\p \Omega_s$. By the comparison principle, we have 
$$0\geq v^s-u\geq \tilde v^s_{j}\ \ \text{in $\Omega_s$.} $$
Note that by the Sobolev inequality \eqref{up}, 
$$\|\tilde v^s_{j}\|_{L^p(\Omega_s)}\leq \|\tilde v^s_{j}\|_{L^p(\Omega)}
   \leq C\cE(\tilde v^s_{j})^{\frac{1}{n+1}}=\left[\int_{\Omega}(-\tilde v^s_{j})g^s_{j}\right]^{\frac{1}{n+1}}.$$
As $j\to\infty$, 
$\tilde{v}^s_{j}$ converges uniformly to a function $\tilde{v}^s$ and 
$$\int_{\Omega}(-\tilde{v}^s_{j})g^s_{j}\to \int_{\Omega}(-\tilde{v}^s)g_s=\int_{\Omega_s}(-\tilde{v}^s)g.$$ 
Therefore, 
$$\|\tilde{v}^s\|_{L^p(\Omega_s)}\leq \Big[\int_{\Omega_s}(-\tilde{v}^s)g\Big]^{\frac{1}{n+1}} . $$
Hence, for $\beta>1$,
\begin{eqnarray*}
\int_{\Omega_s}(-\tilde v^s)g&\leq & C\|\tilde v^s\|_{L^{p^*}(\Omega_s)}  \\
&\leq & C\|\tilde v^s\|_{L^{\beta p^*}(\Omega_s)}|\Omega_s|^{\frac{1}{p^*}\left(1-\frac{1}{\beta}\right)}   \\
&\leq & C\Big[\int_{\Omega_s}(-\tilde{v}^s)g\Big]^{\frac{1}{n+1}}|\Omega_s|^{\frac{1}{p^*}\left(1-\frac{1}{\beta}\right)}.
\end{eqnarray*}
Then we have
$$
 \|\tilde v^s\|_{L^1(\Omega_s)}
  \leq |\Omega_s|^{1-\frac{1}{\beta}}\|\tilde v^s\|_{L^{\beta}(\Omega_s)}
  \leq C|\Omega_s|^{1-\frac{1}{\beta}+\frac{1}{n}\frac{1}{p^*}\left(1-\frac{1}{\beta}\right)}.
$$
Let $\delta=-\frac{1}{\beta}+\frac{1}{n}\frac{1}{p^*}\left(1-\frac{1}{\beta}\right)$. 
We obtain
\begin{equation}\label{decay1}
\|v^s-u\|_{L^1(\Omega_s)}\leq \|\tilde v^s\|_{L^1(\Omega)}\leq C|\Omega_s|^{1+\delta}.
\end{equation}

By \eqref{decay-stab}, 
we can choose $s_0$ large such that $|\Omega_{s_0}|\leq \frac{1}{2}|\Omega|$. 
Denote $s_k:=s_0+\sum_{j=1}^{\infty}2^{-\delta j}$ and $\Omega_k:=\Omega_{s_k}$. 
We claim $|\Omega_{k+1}|\leq \frac{1}{2}|\Omega_k|$. By induction, we assume the inequality holds for $k\leq l$. By \eqref{decay1}, 
\begin{eqnarray*}
|\Omega_{l+1}|&\leq & C\frac{2^{(l+1)\delta}}{t}|\Omega_l|^{1+\delta}  \\
&\leq & C\frac{2^{(l+1)\delta}}{t}\left(\frac{|\Omega_0|}{2^l}\right)^{\delta}|\Omega_l|  \\
&\leq & C2^{\delta}\frac{|\Omega|^{\delta}}{s_0^{\delta}}|\Omega_l|\leq \frac{1}{2}|\Omega_l|.
\end{eqnarray*}
Hence, the claim holds provided $s_0>2^{1+\delta}C|\Omega|^{\delta}$.  
By the claim, we obtain $\big|\bigcap\limits_{l=0}^{\infty} \Omega_l\big|=0$. This implies
$$u-v\leq \Big(s_0+\frac{2^{\delta}}{2^{\delta}-1}\Big)t.$$
The result follows by exchanging $u$ and $v$. 

\vskip 10pt

\noindent{\bf Remark 3.1.}
(1)
Note that in the above proof, we only need the continuity of $g$, but not that of $f$.
By a same argument, we can also obtain the stability for the general case, 
namely when $u, v\in L^{\infty}(\Omega)\bigcap \mathcal{PSH}(\Omega)$.

Indeed, suppose $u$, $v\in L^{\infty}(\Omega)\bigcap \mathcal{PSH}(\Omega)$ and $f$, $g\in L^p$.
Without loss of generality,  we suppose that $u$, $v$ vanish on the boundary. 
For any $\epsilon>0$, let $w$ be the solution to
\begin{align*}
\begin{cases}
(dd^cw)^n=h\,d\mu,\ \ \ & \text{in $\Omega$,}   \\
w=0,           \ \ \ & \text{on $\p \Omega$,}
\end{cases}            
\end{align*}
where we require $h>0$, $h\in C(\Omega)$ and $\|h-g\|_{L^1}^{\frac{1}{n}\frac{\delta}{1+\delta}}\leq \frac{\epsilon}{2}$. Then 
\begin{eqnarray*}
\|u-v\|_\infty&\leq& \|u-w\|_{\infty}+\|w-v\|_{\infty}\\
&\leq& C\Big(\|f-h\|_{L^1}^{\frac{1}{n}\frac{\delta}{1+\delta}}+\|h-g\|_{L^1}^{\frac{1}{n}\frac{\delta}{1+\delta}}\Big)\\
&\leq& C\left(\|f-g\|_{L^1}+\|g-h\|_{L^1}\right)^{\frac{1}{n}\frac{\delta}{1+\delta}}+C\frac{1}{2}\epsilon\\
&\leq& C\|f-g\|_{L^1}^{\frac{1}{n}\frac{\delta}{1+\delta}}+C\epsilon.
\end{eqnarray*}
Letting $\epsilon\to 0$, we obtain the stability in the general case.

 (2)
After obtaining the stability for the case $u\in C^{\infty}(\bar{\Omega})\cap \mathcal{PSH}_0(\Omega)$ 
and $v \in L^{\infty}_{loc}(\Omega)\cap \mathcal{PSH}_0(\Omega) $,  we can obtain \eqref{up} 
for general $u\in L^{\infty}_{loc}(\Omega)\cap \mathcal{PSH}_0(\Omega)$ by the following argument. 
Let $u\in L^{\infty}_{loc}(\Omega)\cap \mathcal{PSH}_0(\Omega)$ and let $f\,d\mu=(dd^cu)^n$ in the sense of measure, 
where $f\in L^1(\Omega)$. Let $f_{j}$ be a smooth approximation of $f$ 
and $u_j$ be the corresponding smooth solutions with vanishing boundary values. 
Then we have 
$$\|u_j\|_{L^p(\Omega)}\leq C\Big(\int_{\Omega}(-u_j)f_j\,d\mu\Big)^{\frac{1}{n+1}}. $$
Taking limits on the both sides, we obtain the inequality for $u$ and $f$. 
Moreover, Theorem \ref{infty} also holds by a similar argument. 

(3)
Note that when $\Omega$ is  smooth and strictly pseudo-convex, 
by the stability (Theorem \ref{stab}), the solution to \eqref{cMA} is continuous. 
In fact, let $f_j\in C^{2,\alpha}(\Omega)$ such that $f_j>0$, $\|f_j-f\|_{L^p}\to 0$, 
and let $\varphi_j\in C^2(\p \Omega)$ such that $\varphi_j\to \varphi$ uniformly.
Then the corresponding smooth solutions $u_j$ converges to $u$ uniformly by letting $v=u_j$ in Theorem \ref{stab}.

\section{The H\"older continuity}

In this section, we give a PDE proof for the H\"older continuity.
We first establish the following estimate.

\begin{lem}\label{stab-holder}
Let $u$, $v$ be  bounded pluri-subharmonic functions in $\Omega$ satisfying $u\geq v$ on $\p \Omega$. 
Assume $(dd^cu)^n=f\,d\mu$ and $0\leq f\in L^p(\Omega)$, $p>1$,
where $\mu$ is the standard Lebesgue measure. 
Then  $\forall\ 0<\delta<\frac{1}{np^*}$,   $\exists\ C>0$,  such that $\forall\ \epsilon>0$,
\beq\label{vu2}
\sup\limits_{\Omega}(v-u)\leq \epsilon+C|\{v-u>\epsilon\}|^{\delta}.
\eeq
\end{lem}

\begin{proof}
The proof is  similar to that of Theorem \ref{stab}. 
Denote $u_{\epsilon}:=u+\epsilon$ and $\Omega_{\epsilon}:=\{v-u_{\epsilon}>0\}$.
Then it suffices to estimate $\|u_{\epsilon}-v\|_{L^{\infty}(\Omega_{\epsilon})}$.

Note that $\Omega_{\epsilon}\Subset\Omega$ and $u_{\epsilon}$ solves 
$$
\Big\{ {\begin{split}
 & (dd^cu_{\epsilon})^n=f\,d\mu \ \ \ \  \text{in $\Omega_{\epsilon}$,} \\[-3pt]
 & \ u_{\epsilon}=v          \hskip55pt    \ \text{on $\pom_{\epsilon}$.}
\end{split}}
$$
Let
$$f_0=
\Big\{ {\begin{split}
f  \ \ &\text{in\ $\Omega_{\epsilon}$} , \\[-3pt]
0  \ \ &\text{on\ $\Omega\setminus\Omega_{\epsilon}$} ,
\end{split}}$$ 
and $u_0$ be the solution to the Dirichlet problem
$$
\Big\{ {\begin{split}
 &(dd^cu_{0})^n=f_0\,d\mu \ \ \ \ \text{in $\Omega$,} \\[-5pt]
 & \ u_{0}=0           \hskip55pt      \ \ \text{on $\pom$.}
\end{split}}
$$
By the comparison principle we have 
$$u_0\leq u_{\epsilon}-v\leq 0 \ \ \text{in\ $\Omega_{\epsilon}$}.$$ 
Hence, by checking the proof of Theorem \ref{weak-infty}, we obtain, similarly,
$$\|u_{\epsilon}-v\|_{L^{\infty}(\Omega_{\epsilon})} 
        \leq \|u_0\|_{L^{\infty}(\Omega_\epsilon)}\leq C|\Omega_{\epsilon}|^{\delta}.$$
\end{proof}

We can now prove the following key estimate without using the capacity theory.

\begin{prop}\label{prep-holder}
Let $u$, $v$ be bounded pluri-subharmonic functions in $\Omega$ 
satisfying $u\geq v$ on $\p \Omega$. 
Assume that $(dd^cu)^n=f\,d\mu$ and $0\leq f\in L^p(\Omega)$, $p>1$, 
where $\mu$ is the standard Lebesgue measure.
Then for $r\geq 1$ and $0\leq \gamma<\frac{r}{np^*+r}$, it holds 
\beq\label{vu}
\sup\limits_{\Omega}(v-u)\leq C\|\max(v-u),0)\|_{L^r(\Omega)}^{\gamma}
\eeq
for a uniform constant $C=C(\gamma,\|f\|_{L^p(\Omega)},\|v\|_{L^{\infty}})>0$. 
\end{prop}

\begin{proof} Note that for any $\epsilon>0$, 
$${\begin{split}
|\{v-u>\epsilon\}| 
 & \leq  \epsilon^{-r}\int_{\{v-u>\epsilon\}}|v-u|^{r}  \\
 &  \leq \epsilon^{-r}\int_\Omega[\max(v-u,0)]^r.
 \end{split}} $$
Let $\epsilon:=\|\max(v-u,0)\|_{L^r(\Omega)}^{\gamma}$, 
 where $\gamma$ is to be determined. 
By Lemma \ref{stab-holder}, we have
\begin{eqnarray}
\sup\limits_{\Omega}(v-u)
  &\leq & \epsilon+C|\{v-u>\epsilon\}|^{\delta} \label{vu3} \\
  &\leq & \|\max(v-u,0)\|_{L^r(\Omega)}^{\gamma}+C\|\max(v-u,0)\|_{L^r(\Omega)}^{\delta r-\delta\gamma r}.  \nonumber
\end{eqnarray}
Choose $\delta<\frac{1}{np^*}$ and  close to $\frac{1}{np^*}$,
and choose $\gamma\leq \delta r-\delta\gamma r$, namely 
$$\gamma\leq \frac{\delta r}{1+\delta r}<\frac{r}{np^*+r}.$$
Then \eqref{vu} follows from \eqref{vu3}.
\end{proof}

\vskip 10pt


For any $\epsilon>0$, we denote $\Omega_\epsilon:=\{x\in \Omega|\, dist(x,\pom)>\epsilon\}$.
Let
\begin{eqnarray*}
u_{\epsilon}(x)&:=&\sup\limits_{|\zeta|\leq \epsilon}u(x+\zeta),\ x\in\Omega_{\epsilon},\\
\hat{u}_{\epsilon}(x)&:=&\bbint_{|\zeta-x|\leq \epsilon}u(\zeta)d\mu,\ x\in\Omega_{\epsilon}.
\end{eqnarray*}
Since $u$ is plurisubharmonic in $\Omega$,  $u_\epsilon$ is a  plurisubharmonic function.
For the H\"older estimate,  it suffices to show there is a uniform constant $C>0$ such that $u_{\epsilon}-u\leq C\epsilon^{\alpha'}$ for some $\alpha'>0$. 
The link between $u_{\epsilon}$ and $\hat{u}_{\epsilon}$ is made by the following lemma.  
\begin{lem}(Lemma 4.1 in \cite{GKZ})\label{interchange}
Given $\alpha\in (0,1)$, the following two conditions are equivalent. 

(1) There exists $\epsilon_0$, $A>0$ such that for any $0<\epsilon\leq \epsilon_0$, 
$$u_{\epsilon}-u\leq A\epsilon^\alpha\ \ \text{on\ $\Omega_\epsilon$}.$$

(2) There exists $\epsilon_1$, $B>0$ such that for any $0<\epsilon\leq \epsilon_1$, 
$$\hat{u}_{\epsilon}-u\leq B\epsilon^\alpha\ \ \text{on\ $\Omega_\epsilon$}.$$
\end{lem}

The following estimate is a generalization of Lemma 4.3 in \cite{GKZ}.

\begin{lem}\label{lapalace-control}
Assume $u\in W^{2, r}(\Omega)$ with $r\geq 1$. Then
for $\epsilon>0$ small enough, we have 
\begin{equation} \label{L54}
\left[\int_{\Omega_{\epsilon}}|\hat{u}_{\epsilon}-u|^r\,d\mu\right]^{\frac{1}{r}}\leq C(n,r)\|\triangle u\|_{L^r(\Omega)}\epsilon^2
\end{equation}
where $C(n,r)>0$ is a uniform constant. 
\end{lem}

\begin{proof}
The proof is essentially contained in \cite{FSX}. 
Note that
\begin{eqnarray*}
\hat{u}_{\epsilon}(z)-u(z)&= & \frac{1}{\omega_{2n}\epsilon^{2n}}\int_{|\zeta-z|\leq \epsilon}u(\zeta)\,d\mu-u(z)    \\
&= & \frac{1}{\omega_{2n}\epsilon^{2n}}\int_0^{\epsilon}t^{2n-1}\left(\int_{|\zeta|=1}(u(z+t\zeta)-u(z))\, d\mu_{S^{2n-1}}\right)\,dt  \\
&= & \frac{1}{\omega_{2n}\epsilon^{2n}}\int_0^{\epsilon}t^{2n-1}\left[\int_{|\zeta|=1}\left(\int_0^t\langle \D u(z+s\zeta),\zeta\rangle\, ds\right)\, d\mu_{S^{2n-1}}\right]\,dt \\
&= & \frac{1}{\omega_{2n}\epsilon^{2n}}\int_0^{\epsilon}t^{2n-1}\left[\int_{0}^t \left(\frac{1}{s}\right)^{2n-1}\left(\int_{|\zeta-z|\leq s}\tr u(\zeta)\, d\mu(\zeta)\right)\, ds\right]\, dt,
\end{eqnarray*}
where $\omega_{2n}$ is the volume of the unit ball in $\mathbb C^n$.
Hence,
\begin{eqnarray*}
&&|\hat{u}_{\epsilon}(z)-u(z)|^r\\&\leq & \frac{\epsilon^{r-1}}{\omega_{2n}^r\epsilon^{2nr}}\int_0^{\epsilon}t^{2n-1}\left[\int_{0}^t \left(\frac{1}{s}\right)^{2n-1}\left(\int_{|\zeta-z|\leq s}\tr u(\zeta)\, d\mu(\zeta)\right)\, ds\right]^r\, dt   \\
&\leq & \frac{\epsilon^{r-1}}{\omega_{2n}^r\epsilon^{2nr}}\int_0^{\epsilon}t^{2nr-1}\left[\int_{0}^t \left(\frac{1}{s}\right)^{(2n-1)r}\left(\int_{|\zeta-z|\leq s}\tr u(\zeta)\, d\mu(\zeta)\right)^r\, ds\right]\, dt  \\
&\leq & \frac{\epsilon^{r-1}}{\omega_{2n}^r\epsilon^{2nr}}\int_0^{\epsilon}t^{2nr-1}\left[\int_{0}^t \left(\frac{1}{s}\right)^{(2n-1)r}(\omega_{2n}s^{2n})^{r-1}\left(\int_{|\zeta-z|\leq s}|\tr u(\zeta)|^r\, d\mu(\zeta)\right)\, ds\right]\, dt   
\end{eqnarray*}
Then by Fubini's theorem,
\begin{eqnarray*}
\int_{\Omega_{\epsilon}}|\hat{u}_{\epsilon}-u|^r\,d\mu&\leq &
  \frac{1}{\omega_{2n}\epsilon^{(2n-1)r+1}}\int_0^\epsilon t^{2nr-1}\left(\int_0^t s^r\|\tr u\|^r_{L^r(\Omega)}\, ds\right)\, dt\\[4pt]
&=& C\|\tr u\|^r_{L^r(\Omega)} \epsilon^{2r}.
\end{eqnarray*}
Then \eqref{L54} follows.
\end{proof}

Note that the function $u_\epsilon$ is not globally defined on $\Omega$. 
However, by $\varphi\in C^{2\alpha}(\partial\Omega)$, 
there exist plurisubharmonic functions $\{\tilde u_\epsilon\}$ which decreases to $u$
as $\epsilon\to 0$ and satisfies \cite{GKZ}
\begin{equation}\label{barrier}
\begin{cases}
\tilde u_\epsilon=u+C\epsilon^\alpha & \text{in}\ \Omega\setminus \Omega_\epsilon;\\
\hat u_\epsilon\leq\tilde u_\epsilon\leq \hat u_\epsilon+C\epsilon^\alpha& \text{in}\ \Omega_\epsilon ,
\end{cases}
\end{equation}
where the constant $C$ is independent of $\epsilon$. 
Then if $u\in W^{2, r}(\Omega)$, by choosing $v=\hat{u}_{\epsilon}$,  
$\gamma< \frac{r}{np^*+r}$ in Proposition \ref{prep-holder}, and using Lemma \ref{lapalace-control}, we have 
\begin{eqnarray}\label{HR}
\sup_{\Omega_\epsilon} (\hat u_\epsilon-u)&\leq& \sup_\Omega(\tilde u_\epsilon-u)+C\epsilon^\alpha \nonumber\\
&\leq& C\|\tilde u_\epsilon-u\|_{L^r}^\gamma+C\epsilon^\alpha\\
&\leq & C\|\triangle u\|_{L^r(\Omega)}^{\gamma}\epsilon^{2\gamma}+C\epsilon^\alpha\nonumber.
\end{eqnarray}
Hence, once we have $u\in W^{2, r}$ for $r\geq 1$, it holds $u\in C^{\alpha'}$ for $\alpha'<\min\{\alpha, \frac{2r}{np^*+r}\}$.

Finally, we show that under the assumption of Theorem \ref{holder}, it holds $u\in W^{2,1}(\Omega)$, i.e., $\triangle u$ has finite mass, and hence $u\in C^{\alpha'}$ for $\alpha'<\min\{\alpha, \frac{2}{np^*+1}\}$. Let $B$ be a ball of $R$ containing $\Om$. We may assume the ball is centered at the origin point. We denote
\begin{align*}
\tilde{f}:=
\begin{cases}
f,\ \     &\text{in}\ \Om,  \\
0,\ \    &\text{in}\ B\setminus \Om,
\end{cases}
\end{align*}
and let $v$ be the solution to the Dirichlet problem 
\begin{align}
\begin{cases}
(dd^cv)^n=\tilde{f}\,d\mu,\ \     &\text{in}\ B,  \\
v=0,\ \    &\text{on}\ \p B.
\end{cases}
\end{align}
Let $K$ be a compact set which satisfies $\Om\subset K\subset B$. 
Consider $b_R:=A(|z|^2-R)$. Choose $A$ sufficiently large such that $(dd^cb_R)^n\geq \tilde{f}\,d\mu$ on $B\setminus K$ and $b_R\leq v$ on $\p (B\setminus K)$. This implies $b_R\leq v$ on $B$. 
Let $h$ be the solution to the Dirichlet problem 
\begin{align}
\begin{cases}
(dd^ch)^n=\epsilon\,d\mu,\ \     &\text{in}\ \Om,  \\
h=-b_R,\ \    &\text{in}\ \p \Om,
\end{cases}
\end{align}
for some $\epsilon>0$. Then the barrier function $b:=h+b_R$ is a smooth subsolution to 
\begin{align}
\begin{cases}
(dd^cu_0)^n=f\,d\mu,\ \     &\text{in}\ \Om,  \\
u_0=0,\ \    &\text{on}\ \p \Om.
\end{cases}
\end{align}
It is clear that $\tr b$ has finite mass.  
By the comparison principle, $\tr u_0$ also has finite mass. Let $w:=u_0+\hat{u}$, where $\hat{u}$ is given in the assumption of Theorem \ref{holder}.  By the assumption, $\tr w$ has finite mass. Note that $w$ is a subsolution to \eqref{cMA}. Again by the comparison principle, $\triangle u$ has finite mass.

\end{document}